\documentclass[12pt]{amsart}
\usepackage{amsmath,amssymb,amsbsy,amsfonts, latexsym,amstext,amsxtra,xcolor,fullpage}
\usepackage[utf8]{inputenc}

\begin{document}

\newtheorem{theorem}{Theorem}
\newtheorem{lemma}[theorem]{Lemma}
\newtheorem{claim}[theorem]{Claim}
\newtheorem{cor}[theorem]{Corollary}
\newtheorem{prop}[theorem]{Proposition}
\newtheorem{example}[theorem]{Example}
\newtheorem{definition}{Definition}
\newtheorem{quest}[theorem]{Question}
\newtheorem{problem}[theorem]{Problem}

\numberwithin{theorem}{section}
\numberwithin{equation}{section}

\def\eps{\varepsilon}
\def\al{\alpha}
\def\be{\beta}
\def\ga{\gamma}
\def\ro{\varrho}
\def\N{\mathbb{N}}
\def \Z{{\mathbb Z}}
\def \Q{{\mathbb Q}}
\def \R{{\mathbb R}}
\def \C{{\mathbb C}}
\def \T{{\mathbb T}}
\def \D{{\mathbb D}}
\newcommand{\cF}{\mathcal{F}}
\newcommand{\cG}{\mathcal{G}}
\newcommand{\cM}{\mathcal{M}}
\newcommand{\cE}{\mathcal{E}}
\newcommand{\cT}{\mathcal{T}}
\newcommand{\dis}{\displaystyle}
\newcommand{\supp}{\textup{supp}}
\newcommand{\dist}{\textup{dist}}
\newcommand{\diam}{\textup{diam}}
\newcommand{\ca}{\textup{cap}}
\newcommand{\ep}{\varepsilon}

\title[Weighted Fekete points on the real line and the unit circle]{Weighted Fekete points on the real line and the unit circle}

\author{Art\= uras Dubickas}
\address{Institute of Mathematics, Faculty of Mathematics and Informatics, Vilnius University, Naugarduko 24,
LT-03225 Vilnius, Lithuania}
\email{arturas.dubickas@mif.vu.lt}

\author{Igor Pritsker}
\address{Department of Mathematics, Oklahoma State University, Stillwater, OK 74078, U.S.A.}
\email{igor@math.okstate.edu}

\begin{abstract}
Weighted Fekete points are defined as those that maximize the weighted version of the Vandermonde determinant over a fixed set. They can also
be viewed as the equilibrium distribution of the unit discrete charges in an external electrostatic field. While these
points have many applications, they are very difficult to find explicitly, and are only known in a few (unweighted) classical cases. We give two
rare explicit examples of weighted Fekete points. The first one is for the weights $w(x)=|x-ai|^{-s}$ on the real line,
with $s\ge 1$ and $a\neq 0,$ while the second is for the weights $w(z)=1/|z-b|$ on the unit circle, with $b\in\R,\ b\neq\pm 1.$
In both cases, we provide solutions of the continuous energy problems with
external fields that express the limit versions of considered weighted Fekete points problems.
\end{abstract}


\maketitle

\section{Weighted Fekete points on the real line} \label{FPR}

For a set of points $Z_n=\{z_k\}_{k=1}^n\subset\C,\ n\ge 2$, the associated Vandermonde determinant is defined by
\[
V(Z_n):=\prod_{1\le j<k\le n} (z_j-z_k).
\]
Fekete \cite{Fe} introduced the notion of the $n$th diameter for a compact set $E\subset\C$ by setting
\[
\delta_n(E):=\max_{Z_n\subset E} |V(Z_n)|^{\frac{2}{n(n-1)}}.
\]
A set of points $\cF_n\subset E$ is called the {\it $n$th Fekete points} of $E$ if it achieves the above maximum $\delta_n(E)$. For example, if $E$ is the unit circumference, then the set of $n$ vertices of any regular $n$-gon inscribed
in $E$ represents the $n$th Fekete points for $E$ (this follows from Hadamard's inequality, see also \cite{BHS,Ran,ST}). Thus a set of Fekete points need not be unique in general. However, if $E=[-1,1]$ then the Fekete points are uniquely identified as the zeros of Legendre polynomial of the second kind. More precisely, let $P_n(x):=\big((x^2-1)^n\big)^{(n)}/(2^n n!)$ be the classical Legendre polynomial of degree $n$, orthogonal on $[-1,1]$ with respect to Lebesgue measure. It is known that then the $n$th Fekete points for $[-1,1]$ are given by the zeros of the polynomial $(x^2-1)P'_{n-1}(x),$ see \cite{Sti1}, \cite{Sch} or p.~140 and Problem 37 in \cite{Sze1}. If $w:E\to[0,\infty)$ is a continuous weight function, then we similarly define the weighted Fekete points by maximizing the weighted Vandermonde determinant:
\begin{align}\label{wFekete}
 \delta_n^w(E) &:=\sup_{Z_n\subset E} \left(\prod_{1\le j<k\le n} |z_j-z_k| w(z_j) w(z_k) \right)^{\frac{2}{n(n-1)}} \\ &= \sup_{Z_n\subset E} \left(|V(Z_n)|^{\frac{2}{n(n-1)}} \prod_{k=1}^{n} w(z_k)^{2/n}\right). \nonumber
\end{align}
In fact, the latter definition allows us to consider closed unbounded sets $E$, provided we have an admissible weight satisfying $\lim_{z\to\infty} |z|w(z) = 0$; see Chapter~III of the book of Saff and Totik \cite{ST} for more details and more general definitions. Some recent results on the asymptotic distribution of
Fekete points (including those on manifolds) can be found, for instance, in \cite{BHS}, \cite{DMN}, \cite{Guo}, \cite{Lev}, \cite{Mont}, and \cite{Vu}.

Our main goal is to describe the weighted Fekete points for the weights $w(x)=|x-ai|^{-s}$ on the real line, with $s\ge 1$ and $a\neq 0.$ Note that for $s<1$ the maximization problem in \eqref{wFekete} gives infinity, and thus is not well defined. (To see this it suffices to take
$n-1$ fixed points $z_i$, $i=1,\dots,n-1$, and $z_n$ arbitrarily large.) The borderline case $s=1$ formally does not fit into the theory developed in \cite{ST}, as this weight is not admissible in the sense of that book. However, it represents the classical case of elliptic potential
theory problem considered in Tsuji \cite[pp. 89-94]{Tsu}. In fact, we are able to give a remarkably explicit solution for that case $s=1$, i.e., for the weight $w(x)=1/|x-ai|$, completely describing these ``elliptic" Fekete points.

\begin{theorem} \label{s=1}
Let $a>0$ and $n \geq 2$. Any weighted Fekete point set $\cF_n^w$ for $w(x)=1/|x-ai|$ on the real line has the form
\begin{equation}\label{wFpoints}
\cF_n^w = \{x_1,\dots,x_n\} = \{a \tan(\gamma + k\pi/n),\quad k=0,1,\ldots,n-1\},
\end{equation}
where $\gamma \in (-\pi/2,-\pi/2+\pi/n)$. Equivalently, the weighted Fekete points are the roots of the monic polynomial
\begin{equation}\label{Fpol}
  F(x) = \frac{1}{2an} \Big( (an-Bi)(x+ai)^n + (an+Bi)(x-ai)^n \Big),
\end{equation}
where
\begin{equation}\label{bbb}
B = -a\sum_{k=0}^{n-1} \tan(\gamma +k\pi/n) = an \cot(n\pi/2+n\gamma).
\end{equation}
Moreover, we have
\begin{equation}\label{wDiam}
\delta_n^w(\R) = \frac{n^{1/(n-1)}}{2a}.
\end{equation}
\end{theorem}

Note that the weighted Fekete points \eqref{wFpoints} are highly non-unique in the above case, since we can take any real $\gamma$ in the interval
$(-\pi/2,-\pi/2+\pi/n)$.

The weight $w(x)=|x-ai|^{-s}$ on the real line is admissible for $s>1$, and all asymptotic results developed in \cite{ST} apply here.
However, in contrast to the case $s=1$ (Theorem~\ref{s=1}), we identify the weighted Fekete points uniquely in this case as roots of some explicit polynomials \eqref{Fpols>1}. In fact, these polynomials
are directly related to Jacobi polynomials defined by
\begin{align}\label{Jacobi}
  P_n^{(\alpha,\beta)}(x) &:= 2^{-n} \sum_{k=0}^n {n+\alpha \choose n-k} {n+\beta \choose k} (x-1)^k (x+1)^{n-k} \\
  &= \frac{(\alpha + \beta + n + 1)_n}{n!\, 2^n} x^n + \ldots, \nonumber
\end{align}
where $(t)_n:=t(t+1)\dots (t+n-1)$ is Pochhammer's symbol (or the rising factorial), and $${t \choose n}:=\frac{t(t-1)\dots(t-n+1)}{n!}$$
is a generalized binomial coefficient.
There are several other equivalent representations for Jacobi polynomials, e.g., via the hypergeometric function, that can be found in
\cite[Ch. 4]{Ismbook}. In the special case $\alpha=\beta=\lambda-1/2$, Jacobi polynomials are also expressible up to a constant factor
(where the constant depends on $n$ and $\lambda$) as Gegenbauer (or ultraspherical) polynomials
$$C^{\lambda}(x):=\sum_{k=0}^{\lfloor n/2 \rfloor}(-1)^k \frac{(\lambda)_{n-k}}{k! (n-2k)!} (2x)^{n-2k}.$$

It is important to note that in our context the parameters $\alpha=\beta=-s(n-1)-1$ not only depend on $n$ but are also outside the classical range
$\alpha>-1$ and $\beta>-1.$ Thus many classical properties such as orthogonality with respect to Jacobi weight on $[-1,1]$ do not hold.
However, we are still able to use the formula for the discriminant of Jacobi polynomial given in (3.4.16) of \cite[Ch. 3]{Ismbook} to
compute the discriminant of our extremal polynomials explicitly, and find the value of the $n$th weighted diameter.

\begin{theorem}\label{s>1}
Suppose that $a>0,\ s>1$ and $n \geq 2$. The weighted Fekete point set $\cF_n^w$ for $w(x)=|x-ai|^{-s}$ on the real line is unique, and is given
by the roots of the polynomial
\begin{align}\label{Fpols>1}
  G(x) := x^n + \sum_{k=1}^{\lfloor n/2 \rfloor} \left((-1)^ka^{2k} {n\choose 2k} \frac{(2k-1)!!}{\prod_{j=1}^{k} (2s(n-1)-2n+2j+1)} \right) x^{n-2k}.
\end{align}
Moreover, we have
\begin{align}\label{GJ}
  G(x) &= \frac{(2ai)^n n!}{(n-2s(n-1)-1)_n} P_n^{(-s(n-1)-1,-s(n-1)-1)}(-ix/a)
\end{align}
and
\begin{align}\label{wDiams>1}
\delta_n^w(\R) &=  \frac{(2a)^{1-2s} (n!)^{2/n}}{\left[(-s(n-1))_n\right]^{2s/n}} \left[(n-2s(n-1)-1)_n\right]^{2(s-1)/n} \\
&\times \left( \prod_{k=1}^{n} k^{k-2n+2} (k-s(n-1)-1)^{2k-2} (n+k-2s(n-1)-2)^{n-k} \right)^{\frac{1}{n(n-1)}}. \nonumber
\end{align}
\end{theorem}
The value of the $n$th weighted diameter is found from the discriminant $\Delta_G$ of $G$, which is expressed via the discriminant of $P_n^{(-s(n-1)-1,-s(n-1)-1)}$, and from the value of $|G(ai)|$ given by
\begin{align}\label{wpart}
|G(ai)| &= \frac{(2a)^n n!}{(n-2s(n-1)-1)_n} P_n^{(-s(n-1)-1,-s(n-1)-1)}(1) \\ &= \frac{(2a)^n n!}{(n-2s(n-1)-1)_n} \frac{(-s(n-1))_n}{n!}
= \frac{(2a)^n (-s(n-1))_n}{(n-2s(n-1)-1)_n} \nonumber,
\end{align}
see (4.1.6) in \cite[p.~82]{Ismbook}. Indeed, by \eqref{wFekete} and $|G(ai)|^{-s}=\prod_{k=1}^d |ai-x_k|^{-s}=\prod_{k=1}^n |w(x_k)|$, we obtain
\begin{align}\label{dD}
\delta_n^w(\R) = |V(\cF_n^w)|^{\frac{2}{n(n-1)}} |G(ai)|^{-2s/n} = |\Delta_G|^{\frac{1}{n(n-1)}} |G(ai)|^{-2s/n}.
\end{align}
We will compute $|\Delta_G|$ explicitly in the proof of Theorem \ref{s>1}.

Polynomials \eqref{Fpols>1} arise in our work from the differential equation \eqref{diffeq} below. They appeared in the literature several times
under different names like pseudo-Jacobi, twisted Jacobi or Romanovski-Routh polynomials; see, for instance, \cite{JT},  \cite{koek}, \cite{kwon}, \cite{rap}, \cite{song},
and a recent survey \cite{Wong}. In particular, the polynomials $G_n=G$ of \eqref{Fpols>1}, with $a=1$, satisfy the recurrence relation
\begin{align}\label{recur}
  G_n(x) = x G_{n-1}(x) - \frac{(n-1)(2s(n-1)-n+3)}{(2s(n-1)-2n+3)(2s(n-1)-2n+5)} G_{n-2}(x), \quad n\ge 2,
\end{align}
and certain finite orthogonality relations on the real line with respect to the weight $w(x) = (x^2 + 1)^{-s(n-1)-1},\ x\in(-\infty,\infty).$
Non-Hermitian orthogonality relations for general Jacobi polynomials were studied in \cite{KMFO}.

Note that the problem of finding the weighted Fekete points is equivalent to finding the equilibrium distribution of unit discrete positive
charges located at $Z_n=\{z_k\}_{k=1}^n\subset\R$ in the external electrostatic field $Q(x) = - \log w(x) = s \log|x-ai|.$ Define the discrete
energy $\cE_w(Z_n)$ of these positive charges in presence of the external field $Q$, created by the fixed negative charge of magnitude
$s$ placed at point the $ai$, by
\[
\cE_w(Z_n) := - \frac{2}{n(n-1)}\sum_{1\le j<k\le n} \log|z_j-z_k| + \frac{2s}{n} \sum_{k=1}^{n} \log |z_k-ai|.
\]
It is clear that finding $\delta_n^w(\R)$ is equivalent to minimizing $\cE_w(Z_n)$:
\[
- \log \delta_n^w(\R) =\inf_{Z_n\subset \R} \cE_w(Z_n).
\]

We would like to mention that the results of Theorems \ref{s=1} and \ref{s>1} apply as well to somewhat more general weights $w(x)=c|x-z_0|^{-s}$,
where $c>0$ and $z_0\in\C\setminus\R,$ by using horizontal translation and scaling.

\section{Weighted energy problem on the real line} \label{EPR}

We now consider the matching case of the continuous weighted energy problem on the real line for the weights
\begin{align} \label{2.1}
w(x)=|x-ai|^{-s}, \qquad x \in\R,
\end{align}
where $s\ge 1$ and $a>0.$ The solution of this problem provides a value of the limit for $\delta_n^w(E)$ as $n\to\infty$, as well as the limit distribution
for the weighted Fekete points. A general treatment of potential theory with external fields, or weighted potential theory, is contained in the book of
Saff and Totik \cite{ST}, together with historical remarks and numerous references. Let ${\mathcal M}(\R)$ be the set of positive unit Borel measures
supported on $\R$. For any measure $\mu\in {\mathcal M}(\R)$ and weight $w$ of the form \eqref{2.1}, we define the energy functional
\begin{align} \label{2.2}
I_{w} (\mu)&:= \int\!\!\int \log \dis\frac{1}{|z-t|w(z)w(t)} \ d
\mu(z)d \mu(t) \\ \nonumber &= \int\!\!\int \log \dis
\frac{1}{|z-t|} \ d \mu(z)d \mu(t) - 2 \int \log w(t)\, d\mu(t),
\end{align}
and consider the minimum energy problem
\begin{equation} \label{2.3}
V_{w}:= \dis\inf_{\mu \in {\mathcal M}(\R)} I_{w}(\mu).
\end{equation}
If $s>1$ then Theorem I.1.3 of \cite{ST} yields that  $V_{w}$ is finite, and that there exists a unique equilibrium measure $\mu_{w} \in
{\mathcal{M}}(\R)$ such that $I_{w} (\mu_{w}) = V_{w}$. Thus $\mu_w$ minimizes the energy functional \eqref{2.2} in presence of
the external field $Q(x) = - \log w(x) = s \log|x-ai|.$ Furthermore, for the potential of $\mu_w$ we have
\begin{equation} \label{2.4}
U^{\mu_{w}}(x) + Q(x) \geq F_{w}, \qquad x \in \R,
\end{equation}
and
\begin{equation} \label{2.5}
U^{\mu_{w}}(x) + Q(x) = F_{w}, \qquad x \in S_w,
\end{equation}
where $$U^{\mu_w}(x):=-\dis\int\log|x-t|\,d\mu_w(t),
\>\>\> F_{w}:=V_{w} + \dis\int \log w(t) d\mu_{w}(t) \>\>\> \text{and} \>\>\> S_w:={\rm supp}\, \mu_{w}$$
(see Theorems I.1.3 and I.5.1 in \cite{ST}). We note that \eqref{2.4} and \eqref{2.5} for general weights and sets hold up to a possible
exceptional set of capacity zero, but in our case there is no exceptional set due to continuity of the weight and the associated equilibrium
potential $U^{\mu_{w}}$. The weighted capacity of $\R$ is defined by
\begin{equation} \label{2.6}
\textup{cap}(\R,w):=e^{-V_w}.
\end{equation}

The support $S_w$ plays a crucial role in determining the equilibrium measure $\mu_w$ itself, as well as some other components of
this weighted energy problem. Indeed, if $S_w$ is known, then $\mu_w$ can be found as a unique solution of the equation
\[
\int\log\frac{1}{|x-t|}\,d\mu(t) + s \log|x-ai| = F, \qquad x \in S_w,
\]
where $F$ is a constant (cf. \eqref{2.5} and Theorem 3.3 of \cite[Ch. I]{ST}). For $w$ given by \eqref{2.1}, this equation can be
solved by potential theoretic methods, using balayage techniques, so that $\mu_w$ is expressed as a linear combination of harmonic measures.
In this way, we obtain the following explicit solution of the minimum energy problem.

\begin{theorem} \label{thm2.1}
Let $w$ be defined by \eqref{2.1} with $a=1$ and $s>1$. The weighted equilibrium measure $\mu_w$ is supported on the compact interval
\[
S_w=\left[-\frac{\sqrt{2s-1}}{s-1}, \frac{\sqrt{2s-1}}{s-1}\right],
\]
and is given by
\begin{equation} \label{2.7}
d\mu_w(x) =\frac{\sqrt{2s-1-(s-1)^2x^2}}{\pi (1+x^2)}\, dx,\quad x\in S_w.
\end{equation}

Moreover, the weighted capacity is equal to
\begin{equation} \label{2.8}
\textup{cap}(\R,w) = 2^{2s-2s^2-1} s^{-s^2} (s-1)^{-(s-1)^2} (2s-1)^{(2s-1)^2/2}.
\end{equation}
\end{theorem}

Theorem 1.3 of \cite[Ch. III]{ST}) combined with the above result gives the following immediate consequence.

\begin{cor} \label{cor2.2}
Let $w$ be defined by \eqref{2.1} with $a=1$ and $s>1$. The weighted Fekete points $\{\zeta_k\}_{k=1}^n$, given by the roots of polynomial
\eqref{Fpols>1}, are distributed according to the measure $\mu_w$ of \eqref{2.7}, i.e.,
\begin{align*}
\tau_n := \frac{1}{n} \sum_{k=1}^{n} \delta_{\zeta_k}  \stackrel{*}{\rightarrow} \mu_w \quad \mbox{as }n\to\infty,
\end{align*}
where the above means the weak* convergence of the normalized counting measures $\tau_n.$

Moreover, the weighted $n$th diameters $\delta_n^w(\R)$ satisfy
\begin{equation*}
\lim_{n\to\infty} \delta_n^w(\R) = \textup{cap}(\R,w) = 2^{2s-2s^2-1} s^{-s^2} (s-1)^{-(s-1)^2} (2s-1)^{(2s-1)^2/2}.
\end{equation*}
\end{cor}

In the case $s=1$ and $w(x)=1/|x-i|$, we deal with the classical elliptic energy problem on the real line, see Tsuji \cite[pp. 89-94]{Tsu}.
The weighted (or elliptic) equilibrium measure is known in this case as the arctan distribution, see Theorem 3 in \cite{FPP}.
The fact of weak* convergence of the counting measures for the weighted Fekete points to the arctan distribution was directly observed
in Corollary 10 of \cite{DP}. These facts can be summarized as follows.

\begin{theorem} \label{thm2.3}
Let $w$ be defined by \eqref{2.1} with $a=1$ and $s=1$. Then the weighted equilibrium measure is given by
\begin{align} \label{2.11}
d\mu_w(x) = \frac{dx}{\pi(1+x^2)} ,\ x\in\R,
\end{align}
and the weighted capacity is $\textup{cap}(\R,w)=1/2$. Furthermore, any sequence of the weighted Fekete point sets $\{\zeta_k\}_{k=1}^n$ satisfies
\begin{align} \label{2.12}
\tau_n := \frac{1}{n} \sum_{k=1}^{n} \delta_{\zeta_k}  \stackrel{*}{\rightarrow} \frac{dx}{\pi(1+x^2)} \quad \mbox{as }n\to\infty.
\end{align}
\end{theorem}

\section{Weighted Fekete points on the unit circle} \label{FPC}

We now consider weighted Fekete points for $w(z)=1/|z-b|$, with $b\in\R,\ b\neq\pm 1,$ on the unit circumference $\T:=\{z:|z|=1\}.$

\begin{theorem} \label{FpT}
Let $b\in\R,\ b\neq\pm 1,$ and $n \geq 2$. Any weighted Fekete point set $\cF_n^w$ for $w(z)=1/|z-b|$ on $\T$ is the image of $n$ equally spaced points
on $\T$ under the mapping
\begin{align}\label{map}
\phi(w):=\frac{bw-1}{w-b},
\end{align}
i.e., it has the form
\begin{equation}\label{wFpT}
\cF_n^w = \{\zeta_1,\dots,\zeta_n\} = \{\phi\left(e^{i(\alpha + 2\pi k/n)}\right),\quad k=0,1,\ldots,n-1\},
\end{equation}
where $\alpha\in [0, 2\pi/n)$. Moreover, we have
\begin{equation}\label{wDiamT}
\delta_n^w(\T) = \frac{n^{1/(n-1)}}{|1-b^2|}.
\end{equation}
\end{theorem}

It is clear that the weighted Fekete points problem for a more general weight $w(z)=A/|z-z_0|$ on a circle $C$, with $z_0\not\in C$ and $A>0$, can be reduced to the
above case by scaling and rotating.

\section{Weighted energy problem on the unit circle} \label{EPC}

In principle, the solution of the continuous version of the weighted energy problem on $\T$ for $w(z)=1/|z-b|,$ where $b\in\R,\ b\neq\pm 1,$ can
be obtained by passing to the limit in Theorem \ref{FpT} as $n\to\infty.$ However, we give a direct proof of the following result:

\begin{theorem} \label{thm4.1}
Let $w(z)=1/|z-b|,$ where $b\in\R,\ b\neq\pm 1$. Then the weighted equilibrium measure for this weight on $\T$ is given by
\begin{align} \label{4.1}
d\mu_w(e^{it}) = \frac{|1-b^2|\,dt}{2\pi(1 - 2b\cos t + b^2)} ,\ t\in[0,2\pi),
\end{align}
and the weighted capacity is
\begin{align} \label{4.2}
\textup{cap}(\T,w)=\frac{1}{|1-b^2|}.
\end{align}
Furthermore, any sequence of the weighted Fekete point sets $\{\zeta_k\}_{k=1}^n$ satisfies
\begin{align} \label{4.3}
\tau_n := \frac{1}{n} \sum_{k=1}^{n} \delta_{\zeta_k}  \stackrel{*}{\rightarrow} \frac{|1-b^2|\,dt}{2\pi(1 - 2b\cos t + b^2)} \quad \mbox{on $\T$ as }n\to\infty.
\end{align}
\end{theorem}

It is obvious from \eqref{wDiamT} and \eqref{4.2} that the weighted $n$th diameters $\delta_n^w(\T)$ satisfy
\begin{equation*}
\lim_{n\to\infty} \delta_n^w(\T) = \textup{cap}(\T,w)=\frac{1}{|1-b^2|}.
\end{equation*}

We remark that the weighted energy problem on the unit circle for general weights was studied in \cite{Pri}. Those results can
be  also applied to obtain a version of Theorem \ref{thm4.1}.

\section{Proofs} \label{Proofs}

\subsection{Proofs for Section \ref{FPR}}

We start with an estimate for a certain product of sine functions contained in \cite{DP}. For convenience, we will give a short proof.

\begin{lemma}\label{prodsin}
For any $y_1,\dots,y_n \in (-\pi/2,\pi/2]$, we have
\begin{equation}\label{cg2}
\prod_{1 \leq j<k \leq d} \sin^2 (y_j-y_k) \leq 2^{-n(n-1)} n^n.
\end{equation}
Furthermore, equality in \eqref{cg2} is attained if and only if the set $\{y_1,\dots,y_n\}$ is an arithmetic progression
with difference $\pi/n$.
\end{lemma}

\begin{proof}
Rearranging the elements $y_k$ in ascending order, we may assume that $-\pi/2 < y_1 < y_2 < \dots < y_n \le \pi/2$.
Notice that $$2\sin (y_k-y_j) =|e^{2i y_k} -e^{2i y_j}|$$
for any pair of indices $j<k$ satisfying $1 \leq j<k \leq n$. Hence
$$2^{n(n-1)} \prod_{1 \leq j<k \leq n} \sin^2 (y_j-y_k) =\prod_{1 \leq j<k \leq n} |e^{2iy_k}-e^{2i y_j}|^2.$$
Here, the product on the right hand side is the square of the absolute value of the Vandermonde determinant for $e^{2i y_1}, e^{2i y_2}, \dots, e^{2i y_n}$.
It is well known that the maximum of the latter does not exceed $n^n$, with equality iff the points $e^{2i y_k}$ for $k=1,\dots,n,$ are equally spaced on the unit
circle, by Hadamard's inequality, cf. \cite{BC}. See also \cite{Ger} for an alternative proof of this fact due to Fekete. This implies the assertion of the lemma.
\end{proof}

\begin{proof}[Proof of Theorem~\ref{s=1}]
Since $a>0$, we can write any set of points $X_n=\{x_k\}_{k=1}^n\in\R$ in the form $x_k=a \tan y_k$, where
$y_k \in (-\pi/2,\pi/2)$ for $k=1,\dots,n$. It follows that
$$\prod_{k=1}^n w(x_k)^2 = \prod_{k=1}^n (a^2+a^2 \tan^2 y_k)^{-1} = a^{-2n} \prod_{k=1}^n \cos^2 y_k.$$
The square of the Vandermonde determinant for these points takes the form
\begin{align*}
V(X_n)^2 &= \prod_{1 \leq j < k \leq n} (a \tan y_j -a \tan y_k)^2=
a^{n(n-1)} \prod_{1 \leq j < k \leq n} \frac{\sin^2 (y_j-y_k)}{\cos^2 y_j \cos^2 y_k} \\
&= a^{n(n-1)} \Big(\prod_{k=1}^n \frac{1}{\cos^2 y_k}\Big)^{n-1} \prod_{1 \leq j<k \leq n} \sin^2 (y_j-y_k).
\end{align*}
Hence
$$
V(X_n)^2 \prod_{k=1}^n w(x_k)^{2(n-1)} = a^{-n(n-1)} \prod_{1 \leq j<k \leq n} \sin^2 (y_j-y_k).
$$
Bounding the right hand side by Lemma~\ref{prodsin}, we find that
\begin{equation}\label{loi1}
V(X_n)^2 \prod_{k=1}^n w(x_k)^{2(n-1)} \leq (2a)^{-n(n-1)} n^n.
\end{equation}
By Lemma~\ref{prodsin}, equality in \eqref{loi1} holds iff $\{y_1,\dots,y_n\} \in (-\pi/2,\pi/2)$ is an arithmetic progression with difference $\pi/n$,
so that \eqref{wDiam} follows after taking an appropriate root of both sides. Note that from Lemma~\ref{prodsin} it follows that the weighted Fekete points are given by
$$
\{\zeta_1,\dots,\zeta_n\}=\{a\tan\gamma, a\tan(\gamma+\pi/n), \ldots,a\tan(\gamma+(n-1)\pi/n)\}
$$
for some $\gamma \in (-\pi/2,-\pi/2+\pi/n)$ as claimed in \eqref{wFpoints}.

Set
\begin{equation*}\label{ff1}
R(x):=\prod_{k=0}^{n-1} (x-a \tan (\gamma+\pi k/n)).
\end{equation*}
To prove \eqref{Fpol} we need to show that $R(ax)=F(ax)$. Since $B=an \cot(n\pi/2+n \gamma)$ by \eqref{bbb} and identity (432) of \cite[pp. 80-81]{Jol},
this is equivalent to
$$\prod_{k=0}^{n-1} (x-\tan(\gamma+k\pi/n)) = \frac{\big(1-i\cot\big(\frac{\pi n}{2}+\gamma n\big)\big)(x+i)^n+\big(1+i\cot\big(\frac{\pi n}{2}+\gamma n \big)\big)(x-i)^n}{2}.$$

We will show that this is an identity that holds for each $ x \in \C$ and all $\gamma \in \C$ for which the involved tangent and cotangent functions are defined.
Indeed, both sides are monic polynomials in $x$  of degree $n$, so it suffices to show that the right hand side vanishes at
$x=\tan(\gamma+k\pi/n)$, $k=0,1,\dots,n-1$.
Let us insert $x=\tan(\gamma+k\pi/n)$ into the right hand side and multiply it by $i^{1-d} \sin(n\pi/2+n\gamma) \cos^n (\gamma+k\pi/n)$. Since
$$i\sin(n\pi/2+n\gamma)(1 \mp i\cot(n\pi/2+n\gamma))=i\sin(n\pi/2+n\gamma) \pm \cos(n\pi/2+n\gamma),$$
and
$$i^{-n}\cos^n(\gamma+k\pi/n) \big (\tan(\gamma+k\pi/n)\pm i\big)^n= \big(-i\sin(\gamma+k\pi/n) \pm \cos(\gamma+k\pi/n)\big)^n,$$
it remains to verify that
$$e^{i(n\pi/2+n\gamma)} e^{-i(k\pi+n\gamma)}-e^{-i(n\pi/2+n\gamma)} (-1)^n e^{i(k\pi+n\gamma)}=0.$$
This equality clearly holds for each $k \in \Z$, since its left hand side equals
$$e^{i \pi(n/2-k)}-(-1)^n e^{i\pi (k - n/2)}=e^{i \pi(n-d/2)}
\big(e^{i \pi (n -2k)}-(-1)^n\big) = e^{i\pi (k-n/2)}
\big(e^{i \pi n}-(-1)^n\big)=0.$$
This completes the proof of \eqref{Fpol}.
\end{proof}

The following lemma is extracted from the proofs of Theorems 12 and 13 in \cite{DP}. We give a short proof for completeness.

\begin{lemma}\label{prodsin2}
Let  $a>0$ and $n\in\N.$ If $f$ is a monic polynomial of degree $n$ satisfying the differential equation
\begin{equation}\label{diff10}
(x^2+a^2)f''(x)-\lambda x f'(x)+n(\lambda-n+1)f(x)=0,
\end{equation}
where $\lambda\neq n-1,n,\dots,2n-2$, then
\begin{align}\label{Fpols>1-1}
  f(x) = x^n + \sum_{k=1}^{\lfloor n/2 \rfloor} \left((-1)^ka^{2k} {n\choose 2k} \frac{(2k-1)!!}{\prod_{j=1}^{k} (\lambda-2n+2j+1)} \right) x^{n-2k}.
\end{align}
\end{lemma}

\begin{proof}
We prove the result for $a=1$, as the case of arbitrary $a>0$ is then immediate by scaling. Suppose that
$$f(x)=x^n+c_{n-1}x^{n-1}+\dots+c_0.$$
Substituting $f(x)$ into \eqref{diff10} with $a=1$, we obtain
\begin{align*}
 \sum_{k=0}^{n} k(k-1) c_k x^k + \sum_{k=0}^{n} k(k-1) c_k x^{k-2} - \lambda \sum_{k=0}^{n} k c_k x^k +n (\lambda - n + 1) \sum_{k=0}^{n} c_k x^k = 0.
\end{align*}
By considering the coefficient for $x^{n-1}$, one gets
$$(n-1)(n-2)c_{n-1}-\lambda (n-1)c_{n-1}+n(\lambda -n+1)c_{n-1}= (\lambda-2n+2) c_{n-1} = 0,$$
which implies $c_{n-1} = 0$. Note also that, by changing $k$ to $k+2$, we can rewrite
the second sum on the left in the form $\sum_{k=0}^{n-2} (k+2)(k+1) c_{k+2} x^k$. Evaluating coefficients for $x^k$, $k=0,1,\dots,n-2$, we find that
$$k(k-1)c_k+(k+2)(k+1)c_{k+2}-\lambda k c_k +n(\lambda-n+1) c_k=0.$$
By the identity $k(k-1)-\lambda k+n(\lambda-n+1)=(n-k)(\lambda-n-k+1),$ this leads to
$$(k+1)(k+2) c_{k+2} +(n-k)(\lambda-n-k+1) c_k = 0$$
for each $k=0,\dots,n-2$. Hence
\begin{align*}
    c_k = \frac{(k+1)(k+2)}{(n-k)(n+k-1-\lambda)} c_{k+2}, \quad k=0,\dots,n-2.
\end{align*}
Applying the latter relation iteratively, with initial value $c_{n-1}=0$, one can easily see that $c_{n-2k-1}=0$
for $k=0,\ldots,\lfloor (n-1)/2 \rfloor$. Similarly, applying it with initial value $c_n=1$, we find that
\begin{equation*}\label{ckk}
    c_{n-2k} = {n\choose 2k} \prod_{j=1}^{k} \frac{2j-1}{2n-2j-1-\lambda}
    = (-1)^k {n\choose 2k} \frac{(2k-1)!!}{\prod_{j=1}^k (\lambda-2n+2j+1)}
\end{equation*}
for $k=1,\ldots,\lfloor n/2 \rfloor$.
\end{proof}

The following formula for the discriminant of Jacobi polynomials can be found in (3.4.16) of \cite[p. 69]{Ismbook}, where it was derived
under the assumption that $\alpha>-1$ and $\beta>-1$. We will explain here why for each $n \geq 2$ it holds for generic $\alpha,\beta\in\C$
such that the leading coefficient of \eqref{Jacobi} does not vanish.

\begin{lemma}\label{Jdiscr}
Let $P_n^{(\alpha,\beta)}$ be the general Jacobi polynomial defined in \eqref{Jacobi} for $\alpha,\beta\in\C$ and some fixed $n \geq 2$. If
$\alpha+\beta\neq -n-k,\ k=1,\ldots,n,$ then the discriminant of $P_n^{(\alpha,\beta)}$ is given by
\begin{align}\label{Jdis}
  \Delta_{P_n^{(\alpha,\beta)}}= 2^{-n(n-1)}\prod_{k=1}^{n} k^{k-2n+2} (k+\alpha)^{k-1} (k+\beta)^{k-1} (n+k+\alpha+\beta)^{n-k}.
\end{align}
\end{lemma}

\begin{proof}
We first remark that if $\alpha+\beta = -n-k$ for some $k \in \{1,\dots,n\}$, then the leading coefficient of $P_n^{(\alpha,\beta)}$ is zero by \eqref{Jacobi}, and the discriminant formula can be interpreted as giving zero value. It is more common, however, to apply the definition of discriminant using the actual degree and the highest non-zero coefficient of $P_n^{(\alpha,\beta)}$. Thus we avoid this ambiguous case in the statement of the lemma.

Note that the coefficients of $P_n^{(\alpha,\beta)}$ are polynomials in $\alpha$ and $\beta$, as defined by the generalized binomial coefficients
\[
{n+\alpha \choose n-k} = \frac{(n+\alpha)(n+\alpha-1)\ldots(\alpha+k+1)}{(n-k)!}
\]
and
\[
{n+\beta \choose k} = \frac{(n+\beta)(n+\beta-1)\ldots(n+\beta-k+1)}{k!}.
\]
Since the discriminant $\Delta_{P_n^{(\alpha,\beta)}}$ is a polynomial in the coefficients of $P_n^{(\alpha,\beta)}$, it is also a polynomial
in $\alpha$ and $\beta$. Note that the right hand side of \eqref{Jdis} is also a polynomial in $\alpha$ and $\beta$. Thus we have two polynomials in
$\alpha$ and $\beta$ that coincide for $\alpha>-1$ and $\beta>-1$ by (3.4.16) of \cite[p. 69]{Ismbook}. Hence they must coincide for all
$\alpha,\beta\in\C$ by the uniqueness results for holomorphic functions. This completes the proof of the lemma.
\end{proof}

\begin{proof}[Proof of Theorem~\ref{s>1}]
We follow ideas similar to those of Stieltjes \cite{Sti1}-\cite{Sti3} and Schur \cite{Sch},
see also \cite[Section 6.7]{Sze1}, and show that the polynomials with roots given by the
weighted Fekete points must satisfy a second order differential equation.
It turns out that this equation has a unique polynomial solution as claimed in Lemma~\ref{prodsin2}.
Then we relate this solution to Jacobi polynomials, and compute its discriminant by using Lemma \ref{Jdiscr}.

Indeed, consider the equivalent logarithmic version of the maximization problem for
\[
g(x_1,\ldots,x_n)= n(n-1)\log \delta_n^w(\R) = \sum_{1\le j<k\le n} \log(x_j-x_k)^2 - s(n-1) \sum_{k=1}^{n} \log(x_k^2+a^2),
\]
where $x_1,\dots,x_n$ are distinct real numbers. Assuming that $g$  has a local maximum, from $\nabla g = 0$ we obtain
\[
\sum_{j\neq k} \frac{2}{x_k-x_j} - \frac{2s(n-1) x_k}{x_k^2+a^2}= 0, \quad k=1,\ldots,n.
\]
In terms of the polynomial $f(x):=\prod_{k=1}^n (x-x_k)$ the latter can be written in the form
\[
 \frac{f''(x_k)}{f'(x_k)} -\frac{2s(n-1) x_k}{x_k^2+a^2} = 0, \quad k=1,\ldots,n,
\]
or, equivalently,
\[
(x_k^2+a^2)f''(x_k) - 2s(n-1) x_kf'(x_k) = 0, \quad k=1,\ldots,n.
\]
Since $(x^2+a^2)f''(x) - 2s(n-1) xf'(x)$ is a polynomial of degree $n$ that vanishes at $n$ distinct points $\{x_k\}_{k=1}^n$, it must be a constant multiple of $f(x).$
Thus we arrive at the differential equation
\[
(x^2+a^2)f''(x) - 2s(n-1) xf'(x) = cf(x),
\]
where $c\in\R.$ Equating the leading coefficients of polynomials on both sides gives
\[
n(n-1) - 2sn(n-1) = c.
\]
The differential equation for $f$ takes the form
\begin{align}\label{diffeq}
(x^2+a^2)f''(x) - 2s(n-1) xf'(x) + n(2s(n-1)-n+1)f(x) = 0.
\end{align}
Since \eqref{diffeq} is \eqref{diff10} with $\lambda=2s(n-1)>2n-2$, we can apply Lemma~\ref{prodsin2} asserting that
\eqref{Fpols>1-1} is the only polynomial solution of this differential equation.
Letting $\lambda=2s(n-1)$ in the polynomial \eqref{Fpols>1-1}, we obtain the polynomial \eqref{Fpols>1}.

We next prove \eqref{GJ} by relating \eqref{diffeq} to the differential equation satisfied by the Jacobi polynomial  $P_n^{(\alpha,\beta)}(t)$:
\[
(1-t^2) u''(t) + (\beta - \alpha - t(\alpha + \beta + 2)) u'(t) + n (n + \alpha + \beta + 1) u(t) = 0,
\]
where in the classical case it is assumed that $ \alpha,\beta>-1,$ see (4.2.6) in  \cite[p. 83]{Ismbook}. We put
\begin{equation}\label{Jacob2}
\alpha = \beta = - s(n-1) - 1,
\end{equation}
so that the equation becomes
\begin{align}\label{Jdiffeq}
(1-t^2) u''(t) + 2s(n-1) t u'(t) + n (n - 2s(n-1) - 1) u(t) = 0,
\end{align}
and $u(t)=P_n^{(-s(n-1)-1,-s(n-1)-1)}(t)$ satisfies it provided $s<0$, which corresponds to the classical range $\alpha,\beta>-1.$
Since in our case $s>1$, we want to show that $P_n^{(-s(n-1)-1,-s(n-1)-1)}(t)$ is a solution of \eqref{Jdiffeq} for any $s\in\R$.

Indeed, for any fixed $n \geq 2$ the coefficients of $P_n^{(-s(n-1)-1,-s(n-1)-1)}(t)$ are polynomials in the parameter $s$ by \eqref{Jacobi} and \eqref{Jacob2}. Fix any $t\in\R$
and consider the polynomial $p(s)$ given by the left hand side of \eqref{Jdiffeq} for $u(t)=P_n^{(-s(n-1)-1,-s(n-1)-1)}(t)$. Since
this polynomial vanishes for all $s<0$, it must vanish identically for all $s\in\R$, and this holds for any $t\in\R$. Setting $t=-ix/a$
and
\begin{align}\label{fJ}
  f(x) := \frac{(2ai)^n n!}{(n-2s(n-1)-1)_n} P_n^{(-s(n-1)-1,-s(n-1)-1)}(-ix/a) = C u(-ix/a),
\end{align}
where $(n-2s(n-1)-1)_n \ne 0$ by $s>1$,
we observe that $f(x)$ is a monic polynomial of degree $n$ by \eqref{Jacobi} and \eqref{Jacob2}.  It is clear that
\[
f'(x) = -\frac{iC}{a} u'(t) \quad\mbox{and}\quad f''(x) = -\frac{C}{a^2} u''(t).
\]
The above relations, together with $t=-ix/a$, transform \eqref{Jdiffeq} into \eqref{diffeq} for this choice of $f$. Hence $f=G$ by
 Lemma~\ref{prodsin2}, and \eqref{GJ} follows.

In the remaining part of this proof, we will derive \eqref{wDiams>1} from \eqref{wpart} and \eqref{dD}. If $P$ is any
polynomial of exact degree $n$ and $Q(x):=cP(-ix/a)$, where $a>0$
and $c\neq 0$ are constants, then the definition of discriminant
readily gives that
\[
|\Delta_Q| = \frac{|c|^{2n-2}}{a^{n(n-1)}} |\Delta_P|.
\]
Applying this relation to \eqref{GJ} with $Q=G$, $P=P_n^{(-s(n-1)-1,-s(n-1)-1)}$, and
$$
c=\frac{(2ai)^n n!}{(n-2s(n-1)-1)_n},
$$
we derive that
\begin{align}\label{discr}
|\Delta_G|^{\frac{1}{n(n-1)}} = 4a \left(\frac{n!}{(n-2s(n-1)-1)_n}\right)^{2/n} |\Delta_P|^{\frac{1}{n(n-1)}}.
\end{align}
Now, Lemma \ref{Jdiscr}, with $\alpha, \beta$ as in \eqref{Jacob2}, gives the formula for the discriminant of the polynomial $P_n^{(-s(n-1)-1,-s(n-1)-1)}$. Inserting that formula into  the right hand side of \eqref{discr} we deduce that
$$
|\Delta_G|^{\frac{1}{n(n-1)}}=2a\left[\frac{n!}{(n-2s(n-1)-1)_n}\right]^{2/n} U,
$$
with
$$
U:=\left( \prod_{k=1}^{n} k^{k-2n+2} (k-s(n-1)-1)^{2k-2} (n+k-2s(n-1)-2)^{n-k} \right)^{\frac{1}{n(n-1)}}.
$$
Inserting this into \eqref{dD} and using \eqref{wpart}, we find that
\begin{align*}
\delta_n^w(\R) & = 2a\left[\frac{n!}{(n-2s(n-1)-1)_n}\right]^{2/n} U
\left(\frac{(2a)^{n}(-s(n-1))_n}{(n-2s(n-1)-1)_n}\right)^{-2s/n}
\\&=\frac{(2a)^{1-2s} (n!)^{2/n}}{\left[(-s(n-1))_n\right]^{2s/n}} \left[(n-2s(n-1)-1)_n\right]^{2(s-1)/n} U,
\end{align*}
which is
\eqref{wDiams>1}.
\end{proof}

\subsection{Proofs for Section \ref{EPR}}

The weighted equilibrium measures for our minimum energy problems turn out to be linear combinations of harmonic measures. For any $r>0$, consider
the domain $\Omega:=\overline{\C}\setminus [-r,r].$ The harmonic measure $\omega_{\Omega}(\xi,\cdot)$ at $\xi\in\Omega$, relative to  $\Omega$,
can be defined as the preimage of the normalized arclength on the unit circle under the conformal mapping $\Phi:\Omega\to\Delta:=\{w:|w|>1\}$
satisfying $\Phi(\xi)=\infty$ and $\Phi'(\xi)>0.$ It is assumed here that $\Phi$ is extended to the boundary of $\Omega$ given by $[-r,r]$ in the standard
way, so that for any Borel set $B\subset\partial\Omega=[-r,r]$ the harmonic measure $\omega_{\Omega}(\xi,B)$ is simply the length of $\Phi(B)$ divided
by $2\pi.$ Thus the harmonic measure is a unit positive Borel measure supported on $[-r,r]$. Alternatively, this harmonic measure can be defined by the balayage
of the unit point mass $\delta_\xi$ from $\Omega$ to $[-r,r]$. More details and background on harmonic measures and balayage can be found in \cite{Ran}
and \cite{ST}.

We begin with some explicit formulas for the necessary harmonic measures.

\begin{lemma}\label{harmeas}
For any $r>0$, let $\Omega:=\overline{\C}\setminus [-r,r].$ The harmonic measures at $\infty$ and $i$, relative to  $\Omega$, are given by
\begin{align} \label{hminf}
d\omega_{\Omega}(\infty,\cdot)(x) &= \frac{dx}{\pi\sqrt{r^2-x^2}},\ x\in(-r,r),
\end{align}
and
\begin{equation} \label{hmi}
d\omega_{\Omega}(i,\cdot)(x) =
\frac{\sqrt{r^2+1} dx}{\pi (1+x^2) \sqrt{r^2-x^2}} ,\ x\in(-r,r).
\end{equation}
\end{lemma}

\begin{proof}
The first equation \eqref{hminf} is well known, as $\omega_{\Omega}(\infty,\cdot)(x)$ is the classical (not weighted) equilibrium measure on $[-r,r]$ given by the arcsin or
Chebyshev distribution, see \cite{Ran}, \cite{ST} and \cite{Tsu}. It remains to find $\omega_{\Omega}(i,\cdot)$ explicitly for $E=[-r,r],\ r>0.$

To do this we will use  the conformal mapping $\Phi$ of $\Omega=\overline{\C}\setminus [-r,r]$ onto $\Delta=\{t\in\overline{\C}:|t|>1\}$ such that $\Phi(i)=\infty.$ As we already mentioned, the image of $\omega_{\Omega}(i,\cdot)$ under $\Phi$ is $\omega_{\Delta}(\infty,\cdot) = |dt|/(2\pi)$ supported on $\partial \Delta = \{t\in\C:|t|= 1\}$, see Theorem 4.3.8 on \cite[p.~101]{Ran}. The mapping $\Phi$ is constructed as the composition of two standard conformal mappings. These are   $$w=\Phi_1(z):=\frac{z+\sqrt{z^2-r^2}}{r}$$
that maps $\Omega$ onto $\Delta$ with
$$w_0=\Phi_1(i)=\frac{(\sqrt{r^2+1}+1)i}{r},$$
and
$$t=\Phi_2(w):=\frac{\overline{w_0} w -1}{w-w_0}$$
that is a self map of $\Delta$ sending $w_0$ to infinity. Defining the upper limiting values
$$\Phi_+(x) := \lim_{y\to 0+} \Phi(x+iy),\quad x\in[-r,r],$$
and the lower limiting values
$$\Phi_-(x) := \lim_{y\to 0-} \Phi(x+iy),\quad x\in[-r,r],$$
we obtain
\[
\Phi_\pm(x) = \Phi_2\left(\frac{x \pm i \sqrt{r^2-x^2}}{r}\right),\quad x\in[-r,r].
\]
This generates the following expression for $d\omega_{\Omega}(i,x)$ as preimage of $|dt|/(2\pi)$:
\begin{equation}\label{prema}
\frac{d\omega_{\Omega}(i,x)}{dx} = \frac{|\Phi'_+(x)|+|\Phi'_-(x)|}{2\pi}\,dx.
\end{equation}

Let us denote $u_{\pm}(x):=(x \pm i \sqrt{r^2-x^2})/r$. Using $\Phi_2'(w)=(1-|w_0|^2)/(w-w_0)^2$  we derive that
$$\Phi_{\pm}'(x)= \frac{d u_{\pm}(x)}{dx} \frac{1-|w_0|^2}{(u_{\pm}(x)-w_0)^2} = \left(\frac{1}{r} \mp \frac{ix}{r\sqrt{r^2-x^2}}\right) \frac{-2(\sqrt{r^2+1}+1)}{(x \pm i\sqrt{r^2-x^2} -(\sqrt{r^2+1}+1)i)^2}.$$
Consequently,
\begin{align*}
|\Phi_{\pm}'(x)| &=\frac{1}{\sqrt{r^2-x^2}} \cdot
\frac{2(\sqrt{r^2+1}+1)}{x^2 + \left(\sqrt{r^2-x^2} \mp (\sqrt{r^2+1}+1)\right)^2} \\& = \frac{2(\sqrt{r^2+1}+1)}{\sqrt{r^2-x^2}(2r^2+2+2\sqrt{r^2+1} \mp 2(\sqrt{r^2+1}+1) \sqrt{r^2-x^2})}
\\& =\frac{\sqrt{r^2+1}+1}{\sqrt{r^2-x^2} (r^2+1+\sqrt{r^2+1} \mp (\sqrt{r^2+1}+1)\sqrt{r^2-x^2})} \\&
=\frac{1}{\sqrt{r^2-x^2}(\sqrt{r^2+1} \mp \sqrt{r^2-x^2})}.
\end{align*}
Therefore,
\begin{align*}
|\Phi_{+}'(x)|+|\Phi_{-}'(x)| & =\frac{1}{\sqrt{r^2-x^2}(\sqrt{r^2+1} - \sqrt{r^2-x^2})}+\frac{1}{\sqrt{r^2-x^2}(\sqrt{r^2+1} + \sqrt{r^2-x^2})} \\&=\frac{2\sqrt{r^2+1}}{\sqrt{r^2-x^2}(r^2+1-(r^2-x^2))}=\frac{2\sqrt{r^2+1}}{(1+x^2)\sqrt{r^2-x^2}}.
\end{align*}
Combining this with \eqref{prema}, we derive \eqref{hmi}.
\end{proof}

\begin{proof}[Proof of Theorem~\ref{thm2.1}]
Let $r>0$ be an arbitrary but fixed number, and set $\Omega=\overline{\C}\setminus [-r,r]$ as before. Define the measure
\begin{equation}\label{mumu}
\mu := s \omega_{\Omega}(i,\cdot) - (s-1) \omega_{\Omega}(\infty,\cdot).
\end{equation}

We will show that $\mu_w=\mu$ for an appropriate choice of $r$. Equations \eqref{2.4} and \eqref{2.5} characterize $\mu_w$ in the sense that
if for a positive unit Borel measure $\mu$ supported on $[-r,r]$ one has
\begin{equation} \label{5.1}
U^{\mu}(x) + s \log|x-i| \geq F, \qquad x \in \R,
\end{equation}
and
\begin{equation} \label{5.2}
U^{\mu}(x) + s \log|x-i| = F, \qquad x \in [-r,r],
\end{equation}
where $F$ is a constant, then Theorem 3.3 of \cite[Ch. I]{ST}) implies that $\mu_w=\mu$ and $F_w=F.$

It is clear that the total mass of $\mu$
is one for any $r>0,$ but we need to ensure that $\mu$ is a positive measure. For
\begin{equation}\label{5.3}
r=\frac{\sqrt{2s-1}}{s-1}
\end{equation}
we have $\sqrt{(r^2+1)/(r^2-x^2)}=s/\sqrt{2s-1-(s-1)^2x^2}$.
Thus, by
Lemma \ref{harmeas} combined with \eqref{mumu}, we find the explicit form for $\mu$
\begin{align*}
d\mu(x) = \left(\frac{s^2}{1+x^2}-(s-1)^2\right) \frac{dx}{\pi\sqrt{2s-1-(s-1)^2x^2}} ,\ x\in(-r,r)
\end{align*}
which simplifies to the form stated on the right hand side of \eqref{2.7}.
From that form and $r$ as in \eqref{5.3}, we clearly see that at the endpoints of the interval
$$[-r,r]=\left[-\frac{\sqrt{2s-1}}{s-1},\frac{\sqrt{2s-1}}{s-1}\right]$$ the density of $\mu$ is zero, and therefore $\mu$ is a positive measure on that interval.

We now show that \eqref{5.2} holds for any $r>0$, by using some general properties of the balayage method in potential theory.
The harmonic measure $\omega_{\Omega}(i,\cdot)$ is the balayage of the point mass $\delta_i$ from the domain $\Omega=\overline{\C}\setminus [-r,r]$ onto its boundary $\partial\Omega=[-r,r]$ , see Section II.4 of \cite{ST}. It follows from Theorem 4.4 of \cite[p.~115]{ST} that the logarithmic potential of $\omega_{\Omega}(i,\cdot)$ satisfies
\begin{align}\label{5.3a}
U^{\omega_{\Omega}(i,\cdot)}(x) + \log|x-i| &= U^{\omega_{\Omega}(i,\cdot)}(x) -  U^{\delta_i}(x) = \int g_\Omega(t,\infty)\,d\delta_i(t)= g_\Omega(i,\infty)
\end{align}
for $x \in [-r,r]$, where $g_\Omega(t,\infty)$ is the Green function of $\Omega$ with logarithmic pole at $\infty.$ We also need the well known fact that the potential of
$\omega_{\Omega}(\infty,\cdot)(x)$, which is the equilibrium measure of $[-r,r]$, is equal to Robin's constant on $[-r,r]$ by Frostman's Theorem (see also
Example 3.5 of \cite[p. 45]{ST} for a direct computation):
\begin{align}  \label{5.3b}
U^{\omega_{\Omega}(\infty,\cdot)}(x)= - \log \textup{cap}([-r,r]) = \log (2/r),\ x\in[-r,r].
\end{align}
Combining \eqref{5.3a} with \eqref{5.3b}, we derive that
\begin{align*}
U^{\mu}(x) + s \log|x-i| &= s \left(U^{\omega_{\Omega}(i,\cdot)}(x) + \log|x-i|\right) - (s-1) U^{\omega_{\Omega}(\infty,\cdot)}(x) \\
&= s g_\Omega(i,\infty) + (s-1)  \log (r/2),\ x\in[-r,r],
\end{align*}
and so conclude that \eqref{5.2} is satisfied with the constant
\begin{align}  \label{5.6}
F = s g_\Omega(i,\infty) + (s-1)  \log (r/2).
\end{align}

The next step is to prove that \eqref{5.1} holds with this value of $F$. For that purpose, we connect the potentials of harmonic measures
with Green functions and conformal mappings. In particular, we have
\begin{align} \label{5.4}
U^{\omega_{\Omega}(\infty,\cdot)}(z) =  \log (2/r) - g_\Omega(z,\infty),\ z\in\Omega,
\end{align}
by Theorem III.37 in \cite[p.~82]{Tsu}. From Theorem III.39 in \cite[p.~84]{Tsu} we obtain
\begin{align}  \label{5.5}
g_\Omega(z,\infty) = \log|\Phi_1(z)|,\ z\in\Omega,
\end{align}
where $\Phi_1:\Omega\to\Delta$ is the conformal mapping defined in the proof of Lemma \ref{harmeas}. Theorem 5.1 of \cite[p. 124]{ST} yields
\begin{align*}
U^{\delta_i}(z) - \int g_\Omega(z,t)\,d\delta_i(t) = U^{\omega_{\Omega}(i,\cdot)}(z) - \int g_\Omega(t,\infty)\,d\delta_i(t), \quad z\in\Omega,
\end{align*}
so that
\begin{align} \label{5.7}
U^{\omega_{\Omega}(i,\cdot)}(z) = g_\Omega(i,\infty) - g_\Omega(z,i) - \log|z-i|, \quad z\in\Omega.
\end{align}

Since $U^{\mu}(x) + s \log|x-i|$ is an even function on $\R,$ and since \eqref{5.2} is already proved, it suffices to show that
the estimate in \eqref{5.1} holds for $x\ge r.$ Using \eqref{5.6}, \eqref{5.4} and \eqref{5.7}, for $x\ge r$ we deduce that
\begin{align*}
U^{\mu}(x) + s \log|x-i| &= s \left(U^{\omega_{\Omega}(i,\cdot)}(x) + \log|x-i|\right) - (s-1) U^{\omega_{\Omega}(\infty,\cdot)}(x) \\
&= s g_\Omega(i,\infty) - s g_\Omega(x,i) + (s-1)  \log (r/2) +(s-1) g_\Omega(x,\infty) \\
&= F + (s-1) g_\Omega(x,\infty) - s g_\Omega(x,i).
\end{align*}
Thus \eqref{5.1} reduces to showing the inequality
\begin{equation}\label{prom}
(s-1) g_\Omega(x,\infty) - s g_\Omega(x,i) \ge 0
\end{equation}
 for $x\ge r,$ under assumption \eqref{5.3}.

Once again we pass to the conformal mappings
from Lemma \ref{harmeas} to express the Green functions by \eqref{5.5} and $$g_\Omega(x,i)=\log|\Phi(x)|=\log|\Phi_2(\Phi_1(x))|.$$ Setting
$y=(\Phi_1(x))^2$ for $x\ge r$, and using $w_0=\Phi_1(i)=(\sqrt{r^2+1}+1)i/r=i\sqrt{2s-1}$ by \eqref{5.3}, we deduce that
$$\left|\Phi_2(\Phi_1(x))\right|=\left|\frac{\overline{w_0}\Phi_1(x)-1}{\Phi_1(x)-w_0}\right|=\sqrt{\frac{1+|\Phi_1(x)|^2(2s-1)}{|\Phi_1(x)|^2+2s-1}}=\sqrt{\frac{(2s-1)y+1}{y+2s-1}},$$
and hence
\begin{align*}
g(y) &:= (s-1) g_\Omega(x,\infty) - s g_\Omega(x,i) = (s-1) \log|\Phi_1(x)| - s \log|\Phi_2(\Phi_1(x))| \\
&= \frac{s-1}{2} \log y - \frac{s}{2} \log
\left(\frac{(2s-1)y+1}{y+2s-1}\right).
\end{align*}
Here, $y \ge 1$ as $x \ge r$.
After some simple algebraic transformations, we arrive at the following expression for the derivative of $g$:
\[
g'(y)=\frac {(s-1)( 2s-1)(y-1)^2}{2y((2s-1)y+1)( 2s-1+y)}.
\]
Since $g(1)=0$ and $g'(y)>0$ for $y>1$, this
completes the proof of the inequality $g(y) \ge 0$ for $y \ge 1$, and so that of \eqref{prom} for $x \ge r$. This proves \eqref{5.1}. Thus $\mu_w=\mu$ and
\begin{align*}
F_w = F =  s g_\Omega(i,\infty) + (s-1)  \log (r/2)
\end{align*}
as in \eqref{5.6}.

On the final step of this proof, we compute the weighted Robin constant $V_w=I(\mu_w)$:
\begin{align}  \label{5.8}
V_w &= F_w + \int Q\,d\mu_w = F_w + s^2 \int \log|x-i|\,d\omega_{\Omega}(i,\cdot) - s(s-1) \int \log|x-i|\,d\omega_{\Omega}(\infty,\cdot) \nonumber \\
&= F_w - s^2 U^{\omega_{\Omega}(i,\cdot)}(i) + s(s-1) U^{\omega_{\Omega}(\infty,\cdot)}(i) \nonumber\\
&= F_w - s^2 U^{\omega_{\Omega}(i,\cdot)}(i) + s(s-1) \left(\log (2/r) - g_\Omega(i,\infty)\right),
\end{align}
where we used \eqref{5.4}. Since the left hand side of \eqref{5.7} is continuous at $z=i$, we obtain
\begin{align*}
U^{\omega_{\Omega}(i,\cdot)}(i) = g_\Omega(i,\infty) - \lim_{z\to i} (g_\Omega(z,i) + \log|z-i|).
\end{align*}

Next, using
$$
|w_0|=\frac{\sqrt{r^2+1}+1}{r} \quad \text{and} \quad |\Phi_1'(i)|=\frac{\sqrt{r^2+1}+1}{r\sqrt{r^2+1}},
$$
we compute the limit
\begin{align*}
\lim_{z\to i} (g_\Omega(z,i) + \log|z-i|) &= \lim_{z\to i} \log|\Phi(z)(z-i)| = \lim_{z\to i} \log|\Phi_2\left(\Phi_1(z)\right)(z-i)| \\ &= \log\left| \lim_{w\to w_0} (\overline{w}_0 w - 1)\, \lim_{z\to i} \frac{z-i}{w-w_0} \right|\quad (w=\Phi_1(z)) \\ &= \log \left|\frac{|w_0|^2-1}{\Phi_1'(i)}\right| = \log\left(\frac{2\sqrt{r^2+1}}{r}\right).
\end{align*}
Hence
$$U^{\omega_{\Omega}(i,\cdot)}(i) = g_\Omega(i,\infty) - \log\left(\frac{2\sqrt{r^2+1}}{r}\right).$$
Combining this with \eqref{5.8}, \eqref{5.6} (where $F_w=F$), \eqref{5.5}, which gives $$g_\Omega(i,\infty) = \log|\Phi_1(i)| =\log |w_0|=\log \left(\frac{\sqrt{r^2+1}+1}{r}\right),$$ and finally with \eqref{5.3},  we derive that
\begin{align*}
V_w &= F_w - s^2 \left( g_\Omega(i,\infty) - \log\left(\frac{2\sqrt{r^2+1}}{r}\right) \right) + (s^2-s)\left(\log \left(\frac{2}{r}\right) - g_\Omega(i,\infty)\right) \\
&= (2s-2s^2) g_\Omega(i,\infty) + s^2 \log \left(\frac{2\sqrt{r^2+1}}{r}\right) - (s-1)^2 \log \left(\frac{r}{2}\right) \\
&= (2s-2s^2) \log\left(\frac{\sqrt{r^2+1}+1}{r}\right) + s^2 \log\left(\frac{2\sqrt{r^2+1}}{r}\right) - (s-1)^2 \log\left(\frac{r}{2}\right) \\
&= (2s-2s^2) \log\sqrt{2s-1} + s^2 \log \left(\frac{2s}{\sqrt{2s-1}}\right) - (s-1)^2 \log \left(\frac{\sqrt{2s-1}}{2(s-1)}\right) \\
&= - \frac{(2s-1)^2}{2} \log(2s-1) + (s-1)^2 \log(s-1) + s^2 \log s + (2s^2-2s+1) \log 2.
\end{align*}
Hence \eqref{2.8} follows from \eqref{2.6}.
\end{proof}

\subsection{Proofs for Section \ref{FPC}}

\begin{proof}[Proof of Theorem~\ref{FpT}]
We first observe that the M\"obius mapping $\phi:\T\to\T$ defined in \eqref{map} is a bijection. Hence for any
set of distinct points $Z_n=\{z_k\}_{k=1}^n\subset\T$ there is a set of distinct points $\{w_k\}_{k=1}^n\subset\T$
such that $z_k=\phi(w_k),\ k=1,\ldots,n$. It follows that
\begin{align*}
|z_j-z_k| = \left|\frac{bw_j-1}{w_j-b} - \frac{bw_k-1}{w_k-b}\right| = \frac{|1-b^2| |w_j-w_k|}{|w_j-b| |w_k-b|}
\end{align*}
and
\begin{align*}
|z_k-b| = \left|\frac{bw_k-1}{w_k-b}-b \right|= \frac{|1-b^2|}{|w_k-b|}.
\end{align*}
Hence
\begin{align*}
|V(Z_n)|^2 \prod_{k=1}^n w(z_k)^{2(n-1)} &= \prod_{1 \leq j<k \leq n} |z_j-z_k|^2 \, \prod_{k=1}^n |z_k-b|^{-2(n-1)} \\
&=  |1-b^2|^{-n(n-1)} \prod_{1 \leq j<k \leq n} |w_j-w_k|^2.
\end{align*}
As we already observed in the proof of Lemma~\ref{prodsin}, the product $\prod_{1 \leq j<k \leq n} |w_j-w_k|^2$ does not exceed $n^n$, and takes this largest value if and only if the points $\{w_k\}_{k=1}^n\subset\T$ are equally spaced on the unit
circle (see \cite{BC} or \cite{Ger}). Thus \eqref{wDiamT}
follows from the definition \eqref{wFekete}, whereas
\eqref{wFpT} follows from \eqref{map}.
\end{proof}

\subsection{Proofs for Section \ref{EPC}}

\begin{proof}[Proof of Theorem~\ref{4.1}]
The strategy of this proof is similar to that of the proof for Theorem~\ref{2.1}. However, this proof is much simpler. In fact, we will show that for the Poisson measure
\begin{equation}\label{5.20}
d\mu(e^{it})=\frac{|1-b^2|\,dt}{2\pi(1 - 2b\cos t + b^2)} ,\ t\in[0,2\pi),
\end{equation}
on the unit circle $\T$, we have
\begin{equation} \label{5.21}
U^{\mu}(z) + \log|z-b| = F, \qquad z \in \T,
\end{equation}
where $F$ is a constant. Thus Theorem 3.3 of \cite[Ch. I]{ST}) implies that $\mu_w=\mu$ and $F_w=F.$ This measure $\mu$ is also well known as the harmonic measure and the balayage of the point mass $\delta_b$ from either the unit disk $\D$ or its exterior $\Delta:=\overline{\C}\setminus \overline{\D}$ onto $\T$, see Section II.4 of \cite{ST} and \cite[p. 96]{Ran}.

Let us first consider the case $|b|>1,$ i.e., $b\in\Delta.$ Then $\mu=\omega_{\Delta}(b,\cdot)$, and it follows from Theorem 4.4 of \cite[p.~115]{ST} that the logarithmic potential of $\omega_{\Delta}(b,\cdot)$ satisfies
\begin{align*}
U^{\omega_{\Delta}(b,\cdot)}(z) + \log|z-b| &= U^{\omega_{\Delta}(b,\cdot)}(z) -  U^{\delta_b}(z) = \int g_\Delta(w,\infty)\,d\delta_b(w)= g_\Delta(b,\infty)
\end{align*}
for $z \in \T$. Here, $g_\Delta(w,\infty)=\log|w|$ is the Green function of $\Delta$ with logarithmic pole at $\infty,$ see \cite[p. 109]{ST}. Hence \eqref{5.21} holds in this case with $F=\log|b|.$ We conclude that $\mu_w=\mu$ and $F_w=F=\log|b|.$

It remains to find the minimum energy $V_w$:
\begin{align}  \label{5.22}
V_w = F_w + \int Q\,d\mu_w = \log|b| + \int \log|z-b|\,d\omega_{\Delta}(b,z) = \log|b| - U^{\omega_{\Delta}(b,\cdot)}(b).
\end{align}
In order to evaluate $U^{\omega_{\Delta}(b,\cdot)}(b)$, we use Theorem 5.1 of \cite[p. 124]{ST}:
\begin{align*}
U^{\delta_b}(z) - \int g_\Delta(z,w)\,d\delta_b(w) = U^{\omega_{\Delta}(b,\cdot)}(z) - \int g_\Delta(w,\infty)\,d\delta_b(w), \quad z\in\Delta,
\end{align*}
so that
\begin{align} \label{5.23}
U^{\omega_{\Delta}(b,\cdot)}(z) = g_\Delta(b,\infty) - g_\Delta(z,b) - \log|z-b|, \quad z\in\Delta.
\end{align}
Since the involved Green functions are given by (cf. \cite[p. 109]{ST})
\begin{align*}
g_\Delta(z,\infty)=\log|z| \quad \mbox{and} \quad g_\Delta(z,b)=\log\frac{|1-bz|}{|z-b|}, \quad z\in\Delta,
\end{align*}
from \eqref{5.23} and \eqref{5.22} we derive that
\begin{align*}
V_w &= \log|b| - U^{\omega_{\Delta}(b,\cdot)}(b) = \lim_{z\to b} (g_\Delta(z,b) + \log|z-b|) \\
&=  \lim_{z\to b} \log|1-bz| = \log|1-b^2|.
\end{align*}
Thus \eqref{4.2} follows from \eqref{2.6}.

In the case when $|b|<1$ (and so $b\in\D$) the proof is slightly different. In this case $\mu=\omega_{\D}(b,\cdot)$, and Theorem 4.1 of \cite[p.~110]{ST} gives
\begin{align*}
U^{\omega_{\D}(b,\cdot)}(z) + \log|z-b| &= U^{\omega_{\D}(b,\cdot)}(z) -  U^{\delta_b}(z) = 0, \quad z \in \T,
\end{align*}
which yields $\mu_w=\mu$ and $F_w=F=0.$ Hence
\begin{align}  \label{5.24}
V_w = F_w + \int Q\,d\mu_w = \int \log|z-b|\,d\omega_{\D}(b,z) = - U^{\omega_{\D}(b,\cdot)}(b).
\end{align}
Applying Theorem 5.1 of \cite[p. 124]{ST}, we find that
\begin{align*}
U^{\delta_b}(z) - \int g_\D(z,w)\,d\delta_b(w) = U^{\omega_{\D}(b,\cdot)}(z), \quad z\in\D,
\end{align*}
and
$$
- U^{\omega_{\D}(b,\cdot)}(z) = g_\D(z,b) + \log|z-b|, \quad z\in\D.
$$
Combining the above identity with \eqref{5.24} and the known representation of the Green function (cf. \cite[p. 109]{ST})
\begin{align*}
g_\D(z,b)=\log\frac{|1-bz|}{|z-b|}, \quad z\in\D,
\end{align*}
we derive that
\begin{align*}
V_w &= - U^{\omega_{\D}(b,\cdot)}(b) = \lim_{z\to b} (g_\D(z,b) + \log|z-b|) \\
&=  \lim_{z\to b} \log|1-bz| = \log|1-b^2|.
\end{align*}
Thus, as before, \eqref{2.6} yields \eqref{4.2}.

Finally, \eqref{4.3} follows from \eqref{4.1} and Theorem 1.3 of \cite[Ch. III]{ST}).
\end{proof}

\medskip
{\bf Acknowledgement.}
The research of the first named author was funded by the European Social Fund according to the activity ‘Improvement of researchers’ quali\-fication by implementing
world-class R\&D projects’ of Measure  No. 09.3.3-LMT-K-712-01-0037. Research of the second author was partially supported by NSF via the American Institute of
Mathematics, and by the College of Arts and Sciences of Oklahoma State University.

\end{document}